\documentclass[12pt]{amsart}
\usepackage{amsaddr}
\usepackage[margin=1in]{geometry}
\usepackage{wrapfig, caption}
\usepackage[utf8]{inputenc}
\usepackage{amsmath, graphicx}
\usepackage{color, bbold}
\usepackage{subcaption}
\usepackage{amsthm}
\usepackage{amssymb}
\usepackage{amscd}
\usepackage{mathtools}
\usepackage{hyperref}
\usepackage{tikz}
\usepackage{comment}
\usepackage{cleveref}
\usepackage{cancel}
\usepackage{graphicx}
\usepackage{microtype}
\usetikzlibrary{shapes.geometric}
\usepackage[shortlabels]{enumitem} 
\usepackage[numbers]{natbib}

\usepackage{tikz}
\usepackage{comment}
\usepackage{cleveref}
\usepackage{bm}
\usepackage{enumitem}
\usetikzlibrary{tikzmark, patterns, calc}
\tikzset{
  norm/.style     = {shape=circle, draw},
  blue/.style     = {shape=circle, draw, fill=blue!25},
  high/.style     = {shape=circle, draw, color=red},
  bluehigh/.style = {shape=circle, draw, color=red, fill=blue!25},
  red/.style      = {shape=circle, draw, fill=red!25},
  both/.style     = {shape=circle, draw, fill=violet!35},
  root/.style     = {node, bottom color=red!30},
  env/.style      = {treenode, font=\ttfamily\normalsize},
  dummy/.style    = {circle}
}
\tikzstyle{standard}=[circle, draw=black, fill=white, very thick]
\tikzstyle{standard2}=[circle, draw=black, fill=cyan!15, inner sep=2pt, scale=0.5, very thick]
\tikzstyle{standard3}=[circle, draw=black, fill=green!15, inner sep=2pt, scale=0.6, very thick]
\tikzstyle{blue2}=[circle, draw=black, fill=blue!25, very thick, inner sep=2pt]
\tikzstyle{small}=[circle, draw=black, fill=black, very thick, minimum size=4mm]
\tikzstyle{small2}=[circle, draw=black, fill=white, very thick, minimum size=4mm] 
\tikzstyle{special}=[circle, draw=red!60, fill=red!5, very thick, inner sep=2pt]
\usetikzlibrary{patterns, calc, decorations.pathreplacing}

\newtheorem{theorem}{Theorem}[section]
\newtheorem{lemma}[theorem]{Lemma}
\newtheorem{cor}[theorem]{Corollary}
\newtheorem{prop}[theorem]{Proposition}
\theoremstyle{definition}
\newtheorem{df}[theorem]{Definition}
\newtheorem{rem}[theorem]{Remark}
\newtheorem{ex}[theorem]{Example}

\newtheorem{conj}[theorem]{Conjecture}

\DeclareMathOperator{\lk}{lk}

\newcommand{\set}[1]{\{#1\}}

\title{Neighborhood Complexes of induced $k$-independent graphs}
\author[Y.F. Shen]{Yufeng Shen}
\address{Xi’an Jiaotong University, Xi’an, Shaanxi 710049, China, yufeng\_shen@stu.xjtu.edu.cn}
\author[Z.Y. Song]{Zhiyu Song}
\address{Nankai University, Tianjin 300071, China, 2210655@mail.nankai.edu.cn}
\author[F.L. Yu]{Fenglin Yu}
\address{Peking University, Beijing 100871, China, fenglin@stu.pku.edu.cn}
\author[W.H. Zhou]{Wuhan Zhou}
\address{Peking University, Beijing 100871, China, wuhanzhou@stu.pku.edu.cn}
\author[J.Q. Zhuang]{Jingqi Zhuang}
\address{Fudan University, Shanghai 200433, China, 22300680047@m.fudan.edu.cn}

\date{\today}
\begin{document}
\subjclass{{57M15, 57Q70, 05C69, 05E45}}
\keywords{neighborhood complexes, induced $k$-independent graphs, homotopy, Morse matching}
\begin{abstract}
    This paper is devoted to \emph{the neighborhood complexes} of the induced $k$-independent graphs. Inspired by the surprising correspondence between \emph{total $k$-cut complex} of $n$-cycle $C_n$ and \emph{neighborhood complex} of stable Kneser graph $SG(n,k)$, we anticipate that the homotopy type of total cut complexes may have some relationships with the neighborhood complexes of induced $k$-independent graphs. We investigated the homotopy type of some total cut complexes and neighborhood complexes of some other graphs, using techniques from algebraic topology and discrete Morse theory.
\end{abstract}
\maketitle

\section{Introduction}\label{sec1:intro}
Recently, there are a kind of simplicial complexes, namely graph complexes, playing a significant role in topological combinatorics. A classcical one is the following result, which was proved by Alon-Frankl-Lovász in 1986, using techniques from graph complexes.
\begin{theorem} [\cite{Alon_1986}]  
    \label{Thm 1.1}
    \[
\chi(KG^r(n,k)) = \left\lceil \frac{n-r(k-1)}{r-1} \right\rceil \quad \text{for all } n \geq rk, r \geq 2. 
\]
\end{theorem}
It is a generalization of the Kneser Conjecture (\cite{Kneser_1955}), which states that the chromatic number of the Kneser graph $KG(n,k)$ is $n-2k+2$.

Recall that a subset \( S \subseteq [n] \) is \textbf{\( r \)-stable} if any two elements \( x, y \in S \) satisfy \( r \leq |x-y| \leq n-r \). Ziegler conjectured that this result by Alon-Frankl-Lovász holds even for the sub-hypergraph of \( KG^r(n,k) \) induced by the \( r \)-stable vertices. We denote the set of all \( r \)-stable sets in \( [n] \) of size \( k \) by \(\binom{[n]}{k}_{r\text{-}\mathrm{Stab}}\). The sub-hypergraph of \( KG^r(n,k) \) induced by \(\binom{[n]}{k}_{r\text{-}\mathrm{Stab}}\) is denoted by \( KG^r(n,k)_{r\text{-}\mathrm{Stab}} \).
Similarly to Theorem \ref{Thm 1.1}, Ziegler proposed the following conjecture.
\begin{conj}[Ziegler,2002]\label{Ziegler}
Let \( KG^r(n,k)_{r\text{-}\mathrm{Stab}} \) be the \( r \)-stable Kneser hypergraph.
\[
\chi(KG^r(n,k)_{r\text{-}\mathrm{Stab}}) = \left\lceil \frac{n-r(k-1)}{r-1} \right\rceil \quad \text{for all } n \geq rk, r \geq 2.
\]
\end{conj}
The special case for $r=2$ was proved by A. Bj\"{o}rner and M. de Longueville \cite{Bjorner_2003} in 2003. He inherited Lovász's ideas \cite {Lovasz_1978}, through the concept of the neighborhood complex of a graph and the Borsuk-Ulam theorem, to show that the stable Kneser graphs are spheres up to homotopy by the following theorem.
\begin{theorem}\cite{Bjorner_2003}\label{Bjorner}
    The neighborhood complex $\mathcal{N}(SG(n,k))$ is homotopy equivalent to the $(n-2k)$-sphere $\mathbb{S}^{n-2k}$ for $n\ge 2k\ge2$.
\end{theorem}

In 2024, \cite{Bayer_2024} and \cite{Bayer_2024_02}, Bayer et al. introduced two new families of graph complexes called \emph{cut complexes} and \emph{total cut complexes}. Their work is motivated by a famous theorem of Ralf Fröberg \cite{Froberg_1990} connecting commutative algebra and graph theory through topology. In \cite{Bayer_2024_02}, Bayer et al. identified the homotopy type of the $k$-total cut complex of $C_n$:
\begin{theorem}\cite{Bayer_2024_02}\label{thm:total-cut-cycle} For $n<2k$, $\Delta_k^t(C_n)$ is the void complex and therefore shellable.
For $n\ge 2k\ge 4$,
$\Delta_k^t(C_n)$ is homotopy equivalent to a single sphere in dimension $n-2k$.
\end{theorem}
Until then, it seems that no relation could be built between neighborhood complexes and total cut complexes. However, an observation by Florian Frick helped us realize that the total cut complex of the cycle graph and the neighborhood complex of the Stable Kneser graph $SG(n,k)$ have the same homotopy type. The following theorem was first proved by Mark Denker and Lei Xue, and we provide a formal proof in this article.
\begin{theorem}\label{Motivation-shen}
    Let $SG(n,k)$ be the stable Kneser graph for $n\ge1,k\ge1$. The $k$-total cut complex of the $n$-cycle is a \emph{nerve complex} of a good cover of the neighborhood complex of $SG(n,k)$, namely $\mathcal{N}(SG(n,k))$. Hence by nerve lemma we have $$\Delta_k^t(C_{n})\simeq\mathcal{N}(SG(n,k))$$
\end{theorem}

This surprising connection suggests that the topology of total cut complexes may have some relationships with the neighborhood complexes of the \emph{induced $k$-independent graph}.

In this paper, we first investigate the relationship between the total cut complex of a graph $G$ and the neighborhood complex of its induced $k$-independent graph $H_k$. Given a graph \( G \), construct a new graph \( H_k \), called the induced $k$-independent graph of $G$, whose vertices are independent sets of size \( k \) in \( G \), with edges whenever these sets are disjoint. Notice that if we identify $G$ by $n$-cycle $C_n$, then $H_k$ is exactly $SG(n,k)$, the stable Kenser graph. We show that for many families of graphs, the homotopy type of the neighborhood complex is of interest in its own right.

The structure of this paper is as follows: 

\begin{itemize}
    \item Section \ref{sec2:preli} introduces some basic definitions and results related to simplicial complexes and graphs in general, and to topological tools.
    \item Section \ref{sec4:neighbor} discusses the neighborhood complexes of several families of graphs and computes their homotopy types, such as prisms over complete graphs $G_n$, the circular ladder graphs $CL_n$, and the squared cycle graph $W_{3k+1}$.
\end{itemize}

\vspace{4em}

\section{Preliminaries}\label{sec2:preli}
We begin by recalling definitions of simplicial complexes and Kneser graphs. We refer readers to \cite{TopMeth}, \cite{Jonsson_2007}, and \cite{Lovasz_1978} for more details.
\subsection{Simplicial Complexes}
\begin{df}
A \emph{simplicial complex} $\Delta$ on a set $A$ is a collection of subsets of $A$ satisfying:
\[
\sigma \in \Delta \text{ and } \tau \subseteq \sigma \Rightarrow \tau \in \Delta.
\]
The elements of $\Delta$ are called its \emph{faces} or \emph{simplices}. If $\Delta$ contains no faces, it is called the \emph{void complex}. Otherwise, $\Delta$ always includes the empty set as a face. The \emph{dimension} of a face $\sigma$, denoted $\dim(\sigma)$, equals one less than its cardinality. A \emph{facet} is a maximal face of the complex. The dimension of a simplicial complex, $\dim(\Delta)$, is the maximum dimension of its facets. We will focus on \emph{pure} simplicial complexes, where all facets share the same dimension.
\end{df}

We will make use of the following operations:
\begin{df}\cite{Kozlov_2008}\label{operations}
    Let $\Delta$ be a simplicial complex and $\sigma$ a face of $\Delta$.
    \begin{itemize}
        \item The \emph{join} of two simplicial complexes $\Delta_1$ and $\Delta_2$ with disjoint vertex sets is the complex \begin{center}{$\Delta_1 * \Delta_2= \{\sigma\cup \tau: \sigma\in \Delta_1, \tau\in \Delta_2\}.$}\end{center}
        \item The \emph{cone} $\mathcal{C}_v(\Delta)$, with cone point $v$, over $\Delta$, and the \emph{suspension} of $\Delta$ are the complexes 
\begin{center}{$\mathcal{C}_v(\Delta)=\Delta* \Gamma_1, \ 
\mathrm{susp}(\Delta)=\Delta*\Gamma_2,$}\end{center}
where $\Gamma_1$ is the 0-dimensional simplicial complex $\{v\}$ with one vertex $v\notin \Delta$, and $\Gamma_2$ is the 0-dimensional complex with two vertices $u,v\notin \Delta$. 
\item The \emph{link} of $\sigma$ in $\Delta$ is 
    $$\lk_{\Delta} \sigma= \{\tau \in \Delta \mid \text{$\sigma\cap \tau 
      = \varnothing$, and $\sigma\cup \tau \in \Delta$}\}.$$

    \end{itemize}
\end{df}

\begin{df}
Let $G = (V,E)$ be a graph. A set $S\subseteq V$ is an \emph{independent set} if and only if no pair of vertices in $S$ forms an edge of~$G$. The {\em independence number} $\alpha(G)$ of $G$ is the cardinality of a maximum independent set in~$G$.
\end{df}
\begin{df}
Given a graph $G=(V,E)$.
\begin{itemize}
    \item The \emph{total cut complex} $\Delta_k^t(G)$ is the simplicial complex whose facets are the complements of independent sets of size $k$ in graph $G$.
    \item The \emph{neighborhood complex} of $G$ is the simplicial complex on the vertex set \( V \) whose simplices are given by sets of vertices that have a common neighbor. We will denote the neighborhood complex of a graph $G$ by $\mathcal{N}(G)$.
\end{itemize}
    
\end{df}

\subsection{Kneser Graphs}
\begin{df}
Given $n\ge k\ge 1$.
\begin{itemize}

    \item The \emph{Kneser graph} $KG(n,k)$ is the graph with vertices the $k$-subsets of $[n]:=\set{1,2,\ldots,n}$; two of vertices are joined by an edge iff they are disjoint, see Figure \ref{fig_KG(5,2)} for $KG(5,2)$ as a example.
    
    \item The vertices of the \emph{stable Kneser graph} \( SG(n,k) \) (introduced in 1978 by A.~Schrijver) are independent sets of size $k$ in the $n$-cycle graph $C_n$; two of them are joined by an edge iff they are disjoint. Notice that \( SG(n,k) \) is an induced subgraph of \( KG(n,k) \),and with respect to the chromatic number, it is vertex critical, see \cite{Schrijver_1978}.
 
\end{itemize}
\end{df}
\begin{figure}
    \centering
    \includegraphics[width=0.5\linewidth]{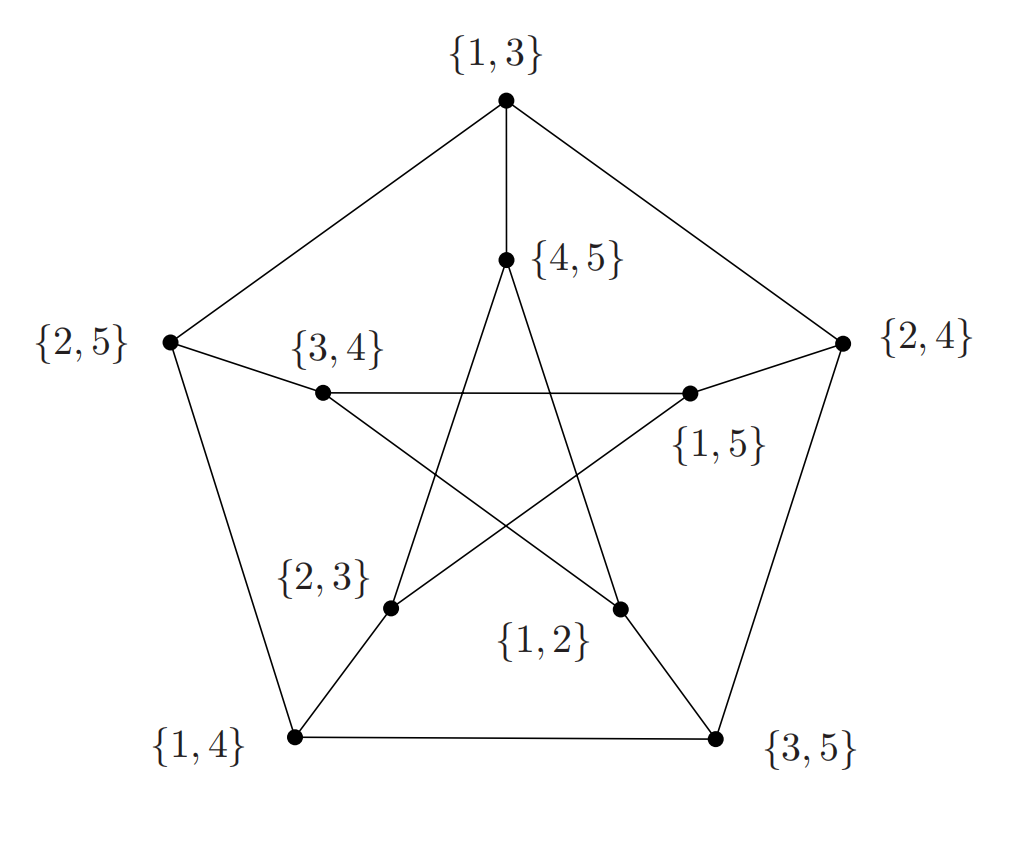}
    \caption{the Kneser graph $KG(5,2)$}
    \label{fig_KG(5,2)}
\end{figure}
The neighborhood complex of \( SG(n,k) \) is hence given by
$$
\mathcal{N}(SG(n,k)) = \left\{ \{I_1, \ldots, I_l\} \subseteq \mathcal{I}_k(C_n) : \exists I  \text{ with } I_j \cap I = \varnothing \ \forall j \right\}
$$
    i.e., the faces of \( \mathcal{N}(SG(n,k)) \) are given by any family of independent \( k \)-sets \( v_i \) in the complement of an independent \( k \)-set \( v \). Sometimes we also refer to a stable $k$-set instead of an independent $k$-set.
\subsection{Topological Tools}
With the basic definitions in place, we now introduce several standard theorems and useful lemmas that will be essential for the proofs that follow.
\begin{df}\cite{BJORNER_2003nerves}\label{df-nerve}
    Let \( \mathcal{A} = \{A_1, A_2, \ldots, A_n\} \) be a family of sets. The \emph{nerve} of \( \mathcal{A} \) records the "intersection pattern" of \( \mathcal{A} \). It is the simplicial complex with the vertex set \([n]\) and with the simplices given by
    \[
\mathfrak{N}(\mathcal{A}) = \left\{ F \subseteq [n]: \bigcap_{i \in F} A_i \neq \varnothing \right\}.
    \]
\end{df}
\begin{lemma}[Nerve Lemma]\cite{TopMeth,BJORNER_2003nerves}\label{Nerve Lemma}
Let $\Delta$ be a simplicial complex and $(\Delta_i)_{i \in I}$ a family of subcomplexes such that $\Delta = \bigcup_{i \in I} \Delta_i$ and every finite nonempty intersection $\Delta_{i_1} \cap \cdots \cap \Delta_{i_s}$ is contractible. Then the nerve complex
\[
\mathfrak{N}(\Delta_i) := 
\left\{ 
\sigma \subseteq I : \sigma \text{ finite}, \bigcap_{i \in \sigma} \Delta_i \neq \varnothing
\right\}
\]
is homotopy equivalent to $\Delta$.
\end{lemma}
\begin{cor}\cite{BJORNER_2003nerves}\label{strong nerve lemma}
Let $\Delta$ be a simplicial complex and $(\Delta_i)_{i \in I}$ a family of subcomplexes such that $\Delta = \bigcup_{i \in I} \Delta_i$ and every finite nonempty intersection $\Delta_{i_1} \cap \cdots \cap \Delta_{i_s}$ is $k-t+1$-connected for $t\geq 1$. Then the nerve complex
\[
\mathfrak{N}(\Delta_i) := 
\left\{ 
\sigma \subseteq I : \sigma \text{ finite}, \bigcap_{i \in \sigma} \Delta_i \neq \varnothing
\right\}
\]
is $k$-connected if and only if $\Delta$ is $k$-connected.
\end{cor}

\begin{df}[Elementary Collapse]\cite{Hatcher_2002}
Let $\Delta$ be a simplicial complex, and let $\sigma, \tau \in \Delta$ be two faces such that:
\begin{enumerate}
    \item $\sigma$ is a \emph{free face} of $\tau$, meaning:
    \begin{itemize}
        \item $\sigma \subsetneq \tau$
        \item $\sigma$ is contained in \emph{no other face} of $\Delta$ except $\tau$
    \end{itemize}
    \item $\dim(\tau) = \dim(\sigma) + 1$
\end{enumerate}
An \emph{elementary collapse} of $\Delta$ is the removal of all faces $\gamma \in \Delta$ such that $\sigma \subseteq \gamma \subseteq \tau$. The resulting subcomplex $\Delta' = \Delta \setminus \{\gamma : \sigma \subseteq \gamma \subseteq \tau\}$ is said to be obtained from $\Delta$ by an elementary collapse.
\end{df}

\begin{df}[Collapse]
A simplicial complex $\Delta$ \emph{collapses} to a subcomplex $\Delta^\prime$ (denoted $\Delta \searrow \Delta'$) if there exists a finite sequence of elementary collapses:
\[
\Delta = \Delta_0 \searrow \Delta_1 \searrow \cdots \searrow \Delta_k = \Delta'
\]
If $\Delta$ collapses to a single vertex, it is called \emph{collapsible}.
\end{df}

\begin{rem}
An elementary collapse is a homotopy equivalence. In particular, if $\Delta \searrow \Delta'$, then $\Delta$ and $\Delta'$ are homotopy equivalent.
\end{rem}

\begin{lemma}[Multicone Lemma]\cite{Bjorner_1999}\label{Multicone Lemma}
Let $\Delta_1 \subseteq \cdots \subseteq \Delta_l = \Delta$ be simplicial complexes, and let $\Delta_0 = \varnothing$. Assume there exist vertices $w_1, \ldots, w_l$ such that for $i = 1, \ldots, l$ the assignment
\[
F \mapsto 
\begin{cases} 
F \cup \{w_i\}, & \text{if } w_i \notin F, \\
F \setminus \{w_i\}, & \text{if } w_i \in F.
\end{cases}
\]
maps $\Delta_i \setminus \Delta_{i-1}$ into itself. Then $\Delta$ is collapsible.
\end{lemma}
We will consider reduced simplicial homology \cite{Hatcher_2002} over the integers \(\mathbb{Z}\). Denote by \(\widetilde{H}_{i}(\Delta)\) the \(i\)-th reduced homology of the simplicial complex \(\Delta\). The following theorem is a fundamental result in algebraic topology that connects the homology (or cohomology) of a space to the homology of its suspension.
\begin{theorem}\cite{TopMeth,Hatcher_2002}\label{suspension-iso}
    Let \(\Delta_{1}\) and \(\Delta_{2}\) be finite complexes. Assume that at least one of \(\widetilde{H}_{p}(\Delta_{1})\), \(\widetilde{H}_{q}(\Delta_{2})\) over \(\mathbb{Z}\) is torsion-free when \(p+q=r-1\). Then the reduced homology of the join \(\Delta_{1}*\Delta_{2}\) in degree \(r\) is given by
\[
\widetilde{H}_{r}(\Delta_{1}*\Delta_{2})\cong\bigoplus_{p+q=r-1}\left(\widetilde {H}_{p}(\Delta_{1})\otimes\widetilde{H}_{q}(\Delta_{2})\right).
\]
In particular, when the appropriate homology groups are torsion-free, the Kunneth Theorem confirms the well-known group isomorphism.

\[
\widetilde{H}_{r}(\operatorname{susp}(\Delta))\cong\widetilde{H}_{r-1}(\Delta).
\]
\end{theorem}

In the end, we refer the reader to \cite{Jonsson_2007} and \cite{Kozlov_2008} for the necessary background on discrete Morse matchings. The particular
case of element matchings is described in \cite{Singh_2021}; a brief summary appears in [\cite{Bayer_2024_02}, Appendix].
\vspace{4em}

\section{Total Cut Complexes versus Neighborhood Complexes}\label{sec3:nerve}
In this section, we prove Theorem \ref {Motivation-shen}, which provides our first motivating example of the correspondence between total cut complexes and neighborhood complexes.
\begin{theorem}\label{motivation-cycles}
    Let $SG(n,k)$ be the stable Kneser graph for $n\ge1,k\ge1$. The $k$-total cut complex of the $n$-cycle is a \textbf{nerve complex} of a good cover of $\mathcal{N}(SG(n,k))$. Hence, by the nerve lemma we have $$\Delta_k^t(C_{n})\simeq\mathcal{N}(SG(n,k))\simeq \mathbb{S}^{n-2k}$$
\end{theorem}
For fixed $n\geq 1, k\geq 1$ and $1\leq i\leq n$, define the subcomplexes
\[
A_i=\{F\subseteq \Delta_v : v \text{ is a vertex of } \mathcal{N}(SG(n,k))  \text{ such that } i\notin v\}
\]
where $\Delta_v=\{w: w\subseteq[n] \text{ stable } k\text{-subset, } v\cap w=\varnothing \}$.

Then we see that $\{A_i\}_{i=1}^n$ forms an open cover of $\mathcal{N}(SG(n,k))$:
\[
\mathcal{N}(SG(n,k))=\bigcup_{i=1}^n A_i
\]
\begin{prop}\label{good-cover}
    For all $n,k\ge 1$, every nonempty intersection $A_{i_1}\cap A_{i_2}\cap\cdots\cap A_{i_m}$ is contractible. 
\end{prop}
\begin{proof}
    Suppose $B=A_{i_1}\cap A_{i_2}\cap\dots\cap A_{i_m}$ is nonempty, we show that it is collapsible using the Multicone Lemma \ref{Multicone Lemma}.
    Without loss of generality, we assume that 
    \[B=\{F\subseteq \Delta_v :v\subseteq [n] \text{ is a stable } k\text{-subset such that }i_1,\dots,i_{m-1},i_m=n\notin v\}.\]
    Since $[n]$ is finite, so is the subcomplex $B$. We may order all stable $k$-subsets $v\subseteq [n]$ that do not  contain $i_1,\dots ,i_m$ lexicographically and label them as $v_1\prec v_2\prec\dots\prec v_l$. For each $i=1,2,\dots,l$, define 
    \[\Gamma_i =\{F\subseteq \Delta_{v_j}: 1\leq j\leq i\},\]
    then $\Gamma_0=\varnothing$, $\Gamma_l=B$ and $\Gamma_i\subseteq \Gamma_{i+1}$. 
    
    For $v_i = \{a_{1},\dots, a_{k}\}\subseteq [n]\setminus\{i_1,\dots,i_m=n\}$, we define  
\[ w_i = \{a_{1}+1, \cdots, a_{k}+1\} \subseteq [n] \]

As presented in Figure \ref{fig:vertices}, it is straightforward to check:
\begin{itemize}
    \item $w_i$ is stable $k$-subset.
    \item $w_i \in \Delta_{v_i}$ since $w_i \cap v_i = \varnothing$. This implies $w_i \in \Gamma_j$ for all $j \geq i$
\end{itemize}

\begin{figure}[htbp]
\centering

\begin{tikzpicture}[scale=1.5, every node/.style={circle, draw, inner sep=2pt}]
    \node [draw,circle,inner sep=2pt,fill=black,label=above:{$1$}] (1) at (0,1){};
    \node [draw,circle,inner sep=2pt,fill=black,label=above right:{$2$}] (2) at (0.3090169944, 0.9510565163){};
    \node [draw,circle,inner sep=2pt,label= above right:{$a_1$}] (3) at (0.5877852523, 0.8090169944){};
    \node [draw,circle,inner sep=2pt,fill=gray,label= right:{$a_1+1$}] (4) at (0.8090169944, 0.5877852523){};
    \node [draw,circle,inner sep=2pt,label= right:{$a_2$}] (5) at (0.9510565163, 0.3090169944){};
    \node [draw,circle,inner sep=2pt,fill=gray,label= right:{$a_2+1$}] (6) at (1,0){};
    \node [draw,circle,inner sep=2pt,fill=black,label= right:{}] (7) at (0.9510565163, -0.3090169944){};
    \node [draw,circle,inner sep=2pt,label= below right:{$a_3$}] (8) at (0.8090169944, -0.5877852523){};
    \node [draw,circle,inner sep=2pt,fill=gray,label= below right:{$a_3+1$}] (9) at (0.5877852523, -0.8090169944){};
    \node [draw=none,fill=none] (10) at (0.3090169944, -0.9510565163){};
    \node [draw,circle,inner sep=2pt,fill=black,label=above left:{$n$}] (11) at (-0.3090169944, 0.9510565163){};
    \node [draw,circle,inner sep=2pt,fill=gray,label=left:{$a_k+1$}] (12) at (-0.5877852523, 0.8090169944){};
    \node [draw,circle,inner sep=2pt,label=left:{$a_k$}] (13) at (-0.8090169944, 0.5877852523){};
    \node [draw,circle,inner sep=2pt,fill=black,label=left:{}] (14) at (-0.9510565163, 0.3090169944){};
    \node [draw=none,fill=none] (15) at (-1,0){};

    \draw (14) -- (13) -- (12) -- (11) -- (1) -- (2) -- (3) -- (4) -- (5) --(6) -- (7) -- (8) -- (9);
    \draw[dashed] (9) -- (10);
    \draw[dashed] (14) -- (15);
\end{tikzpicture}
\caption{The vertices $v_i$ and $w_i$}
    \label{fig:vertices}
\end{figure}
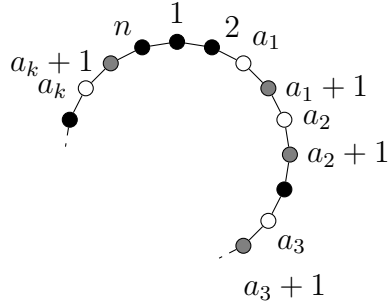

Now it is sufficient to investigate the map defined in the Multicone Lemma.For $i \in \{1, \dots, l\}$ and $F$ a simplex, define  
\[ \phi(F) = 
\begin{cases} 
F \cup \{w_i\}, & \text{if } w_i \notin F, \\
F \setminus \{w_i\}, & \text{if } w_i \in F.
\end{cases} \]

We will verify that  
\(\phi(\Gamma_i \setminus \Gamma_{i-1}) \subseteq \Gamma_i \setminus \Gamma_{i-1}\). If $w_i\notin F$, we map $F$ to $F\cup \{w_i\}$. It is easy to show $F\cup\{w_i\}\in\Gamma_i\setminus\Gamma_{i-1}$. If $w_i\in F$, we map $F$ to $F\setminus\{w_i\}$. In this case $F\setminus\{w_i\}\in \Gamma_i\setminus\Gamma_{i-1}$ for the following reason. Consider the \emph{support} $\operatorname{supp}(F)=\bigcup_{u\in F} u\subseteq [n]$ of $F$. 
The fact that $F\in\Gamma_i\setminus\Gamma_{i-1}$ implies that the lexicographically smallest stable $k$-subset in $[n]\setminus \operatorname{supp}(F)$ is $v_i$. Futhermore, $w_i\in F$ implies that in fact the first $k$ elements of $[n]\setminus\operatorname{supp}(F)$ are given by the set $v_i$. Hence $F\setminus\{w_i\}\in \Gamma_{i-1}$ only if the set $\{a_{i_1},a_{i_1+1},a_{i_2},a_{i_2+1},\dots,a_{i_k},a_{i_k+1}\}$ contains a stable $k$-subset precedes $v_i$ in the lexicographical order. But this is not the case. 

By Multicone Lemma, we conclude that $B=A_{i_1}\bigcap A_{i_2}\bigcap\dots\bigcap A_{i_m}$ is collapsible.

\end{proof}
\begin{prop}\label{nerve-total-cut}
    For $n,k\ge 1$, the nerve complex generated by $\set{A_i}$ is exactly the total cut complex of an $n$-cycle.
\end{prop}
\begin{proof}
The facets of $\Delta_k^t(C_n)$ are complements of independent $k$-subsets of $[n]$. There are exactly $\frac{n}{n-k}\binom{n-k}{k}$ facets.

For an index set $I=\set{i_1,i_2,\dots,i_m}$, we have
 \[
 \bigcap_{i_j\in I}A_{i_j}=\{F\subseteq \Delta_v: v \text{ is a vertex of } \mathcal{N}(SG(n,k))  \text{ such that } i_1,i_2,\cdots,i_{m} \notin v\}
 \]
It is nonempty if and only if there exists a vertex of $\mathcal{N}(H_k)$ such that $i_1,i_2,\dots,i_m\notin v$, which means that $[n]\setminus I$ contains a stable $k$-subset. Meanwhile, this is exactly saying that $I=\set{i_1,\dots,i_m} $ is a face of $\Delta_k^t(C_n)$.
Hence $\Delta_k^t(C_n)$ is a nerve complex $\mathfrak{N}(A_i)$ of open cover $\{A_i\}_{i=1}^{n}$ for $\mathcal{N}(SG(n,k))$.
\end{proof}
Now the \cref{motivation-cycles} can be deduced by Nerve Lemma \ref{Nerve Lemma}, as desired.
\vspace{4em}
\section{Specific Neighborhood Complexes}\label{sec4:neighbor}

Given a graph $G$, construct a new graph $H_k$ whose vertices are independent sets of size $k$ in $G$, with edges whenever these sets are disjoint. We call $H_k$ \emph{induced $k$-independent graph} of $G$.

Based on this surprising connection, we are interested in the relation on the topology between $\mathcal{N}(H_k)$ and $\Delta_k^t(G)$. We show that for many families of graphs, such as prisms over the complete graphs $G_n$, the circular ladder graphs $CL_n$, and the squared cycle graph $W_{3k+1}$, the homotopy type of the neighborhood complex $\mathcal{N}(H_k)$ is of interest in its own right. We begin with the definition of a prism $G_n$ over the complete graph $K_n$.

\begin{df}
    The \emph{prism \( G_n \)} over the complete graph \( K_n \) is the Cartesian product of complete graphs \( K_n \times K_2 \). 
    It has \( 2n \) vertices denoted by \(\{i^+, i^- : 1 \leq i \leq n\}\), and edges $\{i^+ j^+\}_{1 \leq i < j \leq n} ,\{i^- j^-\}_{1 \leq i < j \leq n},\{i^+ i^-\}_{ 1 \leq i \leq n}$.
\end{df}
     When $k>2$, $H_k$ is void. For $k=2$, $H_k$ has vertices $\{i^+j^-\}_{i\neq j}$ with edges whenever these vertices are disjoint. Consider the neighborhood complex $\mathcal{N}(H_2)$, there are $n(n-1)$ facets of dimension $n^2-3n+2$ in the form 
     \[
     F_{i^+j^-}=\{l^+k^-:l\neq i,k\neq j,l\neq k,1\leq l,k\leq n\}\quad 1\leq i\neq j\leq n
     \]

Recall that the $k$-total cut complex of prism $G_n$ has been investigated in \cite{Bayer_2024_02} as follow:
\begin{theorem}\cite{Bayer_2024_02}\label{thm:PrismCliquek=2My2ndDMM2022June21} The $(2n-3)$-dimensional (total) cut complex $\Delta^t_2(G_n), n\ge 2,$ has the homotopy type of a wedge of $(n-1)$ spheres $\mathbb{S}^{2n-4}$ of dimension $2n-4$.
\end{theorem}
\begin{theorem}
    For a prism $G_n$, the neighborhood complex of the induced $2$-independent graph $H_2$, namely $\mathcal{N}(H_2)$, has homotopy type a single sphere of dimension $n-2$.
\end{theorem}
\begin{proof}
The proof is given by the Nerve Lemma \ref{Nerve Lemma}. Let for all $1 \leq i\leq n$ the subcomplexes $U_i$ be defined by:
\[
U_i=\{\sigma \subseteq F_{l^+k^-}: i^+(i+1)^-\in F_{l^+k^-}\}
\]
Obviously, the union of $U_i$'s is $\mathcal{N}(H_2)$. Any $(n-1)$-intersection of $\mathcal{U}=\{U_i:1\leq i\leq n\}$ is nonempty, since \[
\bigcap_{j\in [n]\setminus i}U_j=\{\sigma\subseteq F_{i^+(i+1)^-}\}.
\] 
On the other hand, it is easy to show $\bigcap_{i=1}^{n}U_i=\varnothing$. This shows that the nerve complex generated by $\mathcal{U}$ has $n$ facets: \[
\sigma_i=[n]\setminus i,\quad 1\leq i\leq n
\]
Therefore, it is the boundary of an $(n-1)$-simplex, which is homotopy equivalent to $\mathbb{S}^{n-2}$.

Next, we verify that $\mathcal{U}$ is a good open cover of $\mathcal{N}(H_2)$. This is straightforward since any intersection $U_{i_{1}}\cap U_{i_{2}}\cap \cdots\cap U_{i_{k}}$ is a complex consisting of some facets containing $\{i_1^+(i_1+1)^-\}$. Thus, it is a cone, which is certainly contractible. We conclude by the nerve lemma that
\[
\mathcal{N}(H_2)\simeq \mathfrak{N}(U_i)\simeq\mathbb{S}^{n-2}.
\]
\end{proof}
Since the independent number of a prism $G_n$ is only two, it is sometimes more interesting to consider the subgraph $CL_n$, which has a larger independent number. 
\begin{df}
The \emph{Circular Ladder Graph} $CL_n$ is a subgraph of the prism $G_n$. 
It has $2n$ vertices which we denote by $\{i^{+},i^{-}:1\le i\le n\}$, and edges $\{i^{+}(i+1)^{+}\}_{1\le i\le n}$ (outer cycle), $\{i^{-}(i+1)^{-}\}_{1\le i\le n}$ (inner cycle), $\{i^{+}i^{-}\}_{1\le i\le n}$ (rungs), where indices are taken modulo $n$, see Figure \ref{fig:both_graphs} for examples.
\end{df}
\begin{figure}[htbp]
    \centering
    \begin{subfigure}{0.4\textwidth}
        \centering
        \begin{tikzpicture}[scale=0.8]
            \node[standard2] (node1-) at (0:1.4cm) {$1^-$};
            \node[standard2] (node2-) at (72:1.4cm) {$2^-$};
            \node[standard2] (node3-) at (144:1.4cm) {$3^-$};
            \node[standard2] (node4-) at (216:1.4cm) {$4^-$};
            \node[standard2] (node5-) at (288:1.4cm) {$5^-$};
            
            \node[standard3] (node1+) at (0:3cm) {$1^+$};
            \node[standard3] (node2+) at (72:3cm) {$2^+$};
            \node[standard3] (node3+) at (144:3cm) {$3^+$};
            \node[standard3] (node4+) at (216:3cm) {$4^+$};
            \node[standard3] (node5+) at (288:3cm) {$5^+$};

            \draw[thick, gray] (node1-) -- (node2-) -- (node3-) -- (node4-) -- (node5-) -- (node1-) -- (node3-) -- (node5-) -- (node2-) -- (node4-) -- (node1-);
            \draw[very thick, gray] (node1+) -- (node2+) -- (node3+) -- (node4+) -- (node5+) -- (node1+) -- (node3+) -- (node5+) -- (node2+) -- (node4+) -- (node1+);
            \draw[very thick] (node1+) -- (node1-);
            \draw[very thick] (node2+) -- (node2-);
            \draw[very thick] (node3+) -- (node3-);
            \draw[very thick] (node4+) -- (node4-);
            \draw[very thick] (node5+) -- (node5-);
        \end{tikzpicture}
        \caption{Prism Graph $G_5$}
    \end{subfigure}
    \hspace{1em}
    \begin{subfigure}{0.4\textwidth}
        \centering
        \begin{tikzpicture}[scale=0.8]
            \node[standard2] (node1-) at (0:1.4cm) {$1^-$};
            \node[standard2] (node2-) at (72:1.4cm) {$2^-$};
            \node[standard2] (node3-) at (144:1.4cm) {$3^-$};
            \node[standard2] (node4-) at (216:1.4cm) {$4^-$};
            \node[standard2] (node5-) at (288:1.4cm) {$5^-$};
            
            \node[standard3] (node1+) at (0:3cm) {$1^+$};
            \node[standard3] (node2+) at (72:3cm) {$2^+$};
            \node[standard3] (node3+) at (144:3cm) {$3^+$};
            \node[standard3] (node4+) at (216:3cm) {$4^+$};
            \node[standard3] (node5+) at (288:3cm) {$5^+$};

            \draw[thick, gray] (node1-) -- (node2-) -- (node3-) -- (node4-) -- (node5-) -- (node1-);
            \draw[very thick, gray] (node1+) -- (node2+) -- (node3+) -- (node4+) -- (node5+) -- (node1+);

            \draw[very thick] (node1+) -- (node1-);
            \draw[very thick] (node2+) -- (node2-);
            \draw[very thick] (node3+) -- (node3-);
            \draw[very thick] (node4+) -- (node4-);
            \draw[very thick] (node5+) -- (node5-);
        \end{tikzpicture}
        \caption{Circular Ladder Graph $CL_5$}
    \end{subfigure}
    \caption{Prism graph $G_5$ and Circular Ladder graph $CL_5$}
    \label{fig:both_graphs}
\end{figure}
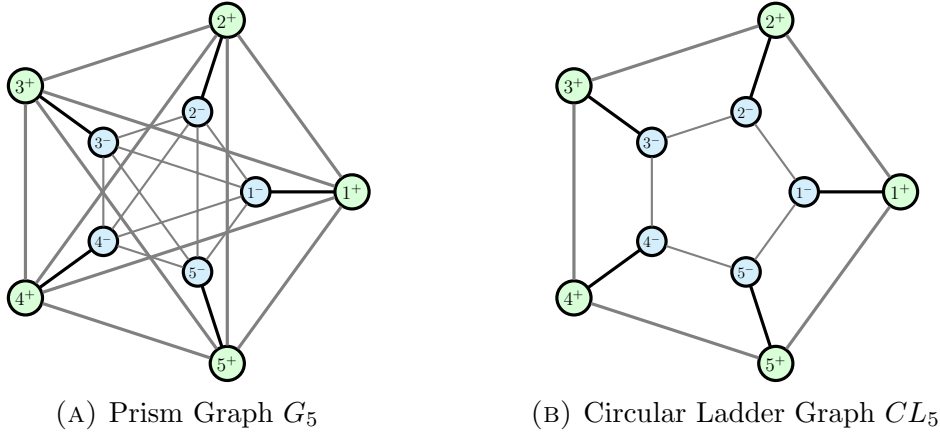
\begin{theorem}\label{CL-total-cut}
    The $n$-dimensional total cut complex $\Delta_{n-1}^{t}(CL_n)$, has the homotopy type of a wedge of $(n-1)$ spheres $\mathbb{S}^{2}$ if $n$ is odd, while it has the homotopy type of a wedge of $(n-1)^2$ spheres $\mathbb{S}^2$ if $n$ is even.
\end{theorem}
\begin{proof}
 For even $n$, define:
    \[ A = \{1^+, 2^-, \dots, n^-\}, \quad B = \{1^-, 2^+, \dots, n^+\} \]
    
    Let $\Delta_A$ be a simplicial complex containing one facet $\{1^+,2^-,\dots,n^-\}$, and $A_{\text{discrete}}$ be a simplicial complex containing $n$ facets: $\{1^+\},\{2^-\},\dots,\{n^-\}$.
    
    The complex decomposes as $\Delta_{n-1}^t(CL_n) = X \cup Y$ where:
    \begin{align*}
        X &= \Delta_A * B_{\text{discrete}} \\
        Y &= \Delta_B * A_{\text{discrete}}
    \end{align*}
    Here $X$, $Y$ are both contractible with $X \cap Y$ being the 1-skeleton of $K_{n,n}$, which is homotopy equivalent to $\bigvee_{(n-1)^2} \mathbb{S}^1$.

    Using the suspension isomorphism \ref{suspension-iso}, we conclude that:
    \begin{equation}
        \Delta_{n-1}^t(CL_n) \simeq \text{susp}(\bigvee_{(n-1)^2} \mathbb{S}^1) \simeq \bigvee_{(n-1)^2} \mathbb{S}^2
    \end{equation}

Now assume $n$ is odd. We construct a Morse matching on the face poset of $\Delta_{n-1}^t(CL_n)$ and show there are exactly $(n-1)$ critical cells of dimension $2$. The critical cells are:
    \[ \sigma_j = \{1^-, j^{+}, j^{-}\} , \quad 2 \leq j \leq n \]
    
    Partition the vertices of $CL_n$ into:
    \[ V^+ = \{1^+, \dots, n^+\}, \quad V^- = \{1^-, \dots, n^-\} \]
    
    The facets of $\Delta_{n-1}^t(CL_n)$ are:
    \[ A^+(i) = \{i^+, i^-, (i+1)^+, (i+2)^-, \dots, (i+n-1)^-\} \]
    \[ A^-(i) = \{i^+, i^-, (i+1)^-, (i+2)^+, \dots, (i+n-1)^+\} \]
    for $1 \leq i \leq n$, where indices are taken modulo $n$.
    
    A subset $\sigma$ is a face iff $\sigma \subseteq A^{\pm}(i)$ for some $1\leq i\leq n$. 
    
    We use an element matching $\mathcal{M}_{1^+}$ first with $1^+$, followed by the element matching $\mathcal{M}_{1^-}$ with $1^-$. After the matching $\mathcal{M}_{1^+}$, the unmatched faces are $\sigma$ such that $1^+\notin \sigma$ and $1^+\cup \sigma$ is NOT a face. Now match with $1^-$. Given a face $\sigma$ unmatched by $\mathcal{M}_{1^+}$, we examine two possible cases that would result in $\sigma$ NOT being matched by $1^-$.
    \begin{enumerate}
        \item $1^-\notin \sigma, 1^+\notin \sigma$ and both $1^+\cup\sigma$ and $1^-\cup \sigma$ are NOT faces. Therefore, we may suppose $\sigma \subset A^{\pm}(i)$ for some $i\neq 1$, otherwise $1^+\cup\sigma$ will be a face. Then other $1^+\in A^{\pm}(i)$ or $1^-\in A^{\pm}(i)$, which leads that one of $1^+\cup\sigma$ and $1^-\cup \sigma$ is a face. This case is thus eliminated.
        \item $1^-\in\sigma$ but $\sigma$ cannot be matched with $\tau=\sigma\setminus 1^-$: This happens only if $\tau$ was already matched by $\mathcal{M}_{1^+}$, that is , $1^+\cup(\sigma\setminus 1^-)$ is a face.
    \end{enumerate}
    
    Hence, the unmatched faces $\sigma$ after $\mathcal{M}_{1^+}$, $\mathcal{M}_{1^-}$ are precisely those in Part (2) above, which we will describe in more detail now.


    Let $\gamma=\sigma\setminus1^-$, the conditions are equivalent to 
    $1^+, 1^-\notin\gamma$, $\gamma\cup1^+ , \gamma\cup 1^-$ are faces and $\gamma\cup 1^+\cup 1^-$ is NOT a face. Obviously $\gamma\cup 1^{\pm}$ is not in $A^{\pm}(1)$. We then see that $\gamma\cup 1^+\subseteq A^-(2k+1)$ or $A^+(2k)$ for some $k\geq 1$ since $n$ is odd.

Therefore, $\gamma^c\setminus1^+$ contains an independent set of size $n-1$. In fact, there are only two possible cases.
\begin{enumerate}
    \item $\{1^-,2^+,\dots,(2k)^+,(2k+2)^-,\dots,n^+\}$;
    \item $\{1^-,2^+,\dots,(2k-1)^-,(2k+1)^+,\dots,n^+\}$. 
\end{enumerate}
Similarly, $\gamma^c\setminus 1^-$ contains either
$\{1^+,2^-,\dots,(2m)^-,(2m+2)^+,\dots,n^-\}$ or
$\{1^+,2^-,\dots,(2m-1)^+,(2m+1)^-,\dots,n^-\}$ for some $m\geq 1$. Without loss of generality, we may assume $k<m$. There are four possible relations between $\gamma\setminus 1^+$ and $\gamma\setminus1^-$. First, we check that:
\begin{itemize}
\item $\gamma^c\setminus 1^+$ contains $\{1^-,2^+,\dots,(2k)^+,(2k+2)^-,\dots,n^+\}$,
\item $\gamma^c\setminus 1^-$ contains $\{1^+,2^-,\dots,(2m)^-,(2m+2)^+,\dots,n^-\}$.
\end{itemize}

   In this case, it turns out that $\gamma^c\setminus(1^+\cup1^-)$ contains $\{2^{\pm},3^{\pm},\cdots,(2k)^{\pm},(2k+1)^+,(2k+2)^-,\cdots,(2m+1)^+,(2m+2)^{\pm},\cdots,n^{\pm}\}$, which deduces that $\gamma^c\setminus(1^+\cup 1^-)$ contains an independent set of size $n-1$. This is a contradiction!

    Hence $k=m$ and since $\gamma\cup 1^+\cup 1^-$ is not a face, we conclude that $\gamma=\{(2k+1)^+,(2k+1)^-\}$ is the only case meeting the requirements.

The same argument shows if 
\begin{itemize}
\item $\gamma^c\setminus 1^+$ contains $\{1^-,2^+,\dots,(2k-1)^-,(2k+1)^+,\dots,n^+\}$,
\item $\gamma^c\setminus 1^-$ contains $\{1^+,2^-,\dots,(2m-1)^+,(2m+1)^-,\dots,n^-\}$.
\end{itemize}
then $\gamma=\{(2k)^+,(2k)^-\}$.

As for the latter two cases, we claim that there is a contradiction. Assume that 
\begin{itemize}
\item $\gamma^c\setminus 1^+$ contains $\{1^-,2^+,\dots,(2k-1)^-,(2k+1)^+,\dots,n^+\}$,
\item $\gamma^c\setminus 1^-$ contains $\{1^+,2^-,\dots,(2m)^-,(2m+2)^+,\dots,n^-\}$.
\end{itemize}
In this case, it turns out that $\gamma^c\setminus(1^+\cup1^-)$ contains $\{2^{\pm},3^{\pm},\cdots,(2k-1)^{\pm},(2k)^-,(2k+1)^+,\cdots,(2m)^-,(2m+1)^+,(2m+2)^{\pm},\cdots,n^{\pm}\}$, which deduces that $\gamma^c\setminus(1^+\cup1^-)$ contains an independent set of size $n-1$, contrary to the fact $\gamma\cup 1^+\cup 1^-$ is not a face. Then $k=m$, but in this condition, $\gamma^c\setminus(1^+\cup1^-)$ contains $\{2^{\pm},3^{\pm},\cdots,(2k-1)^{\pm},(2k)^-,(2k+1)^+,(2k+2)^{\pm},\cdots,n^{\pm}\}$, which is also a contradiction! 
    
    Therefore, we conclude that there are $n-1$ critical cells of our Morse matching, $\sigma=\{1^-,j^+,j^-\}_{2\le j\le n}$, and so combining with a $0$-cell because $\varnothing$ is matched with $\{1^+\}$, the statement about the homotopy type follows.
    \begin{equation}
        \Delta_{n-1}^t(CL_n)\simeq \bigvee_{n-1}\mathbb{S}^2
    \end{equation}
\end{proof}
Now we focus on the homotopy type of the neighborhood complexes of $\mathcal{N}(H_{n-1})$. This problem also needs to be considered in terms of odd and even.
\begin{theorem}\label{CL-neighborhood}
When $n$ is odd, the induced $(n-1)$-independent graph $H_{n-1}$ is isomorphic to $CL_n$ as graphs: $H_{n-1}\cong CL_n$. Then the neighborhood complex $\mathcal{N}(H_{n-1})$ is homeomorphic to the neighborhood complex $\mathcal{N}(CL_n)$ and is homotopy equivalent to $\mathbb{S}^1$.

When $n$ is even ,the neighborhood complex $\mathcal{N}(H_{n-1})$ is homotopy equivalent to $\mathbb{S}^0$.

\end{theorem}
\begin{proof}
    First, assume $n$ is odd. By definition, $H_{n-1}$ has $2n$ vertices:
    \[ V\setminus A^+(i) = \{(i+1)^-, (i+2)^+, \dots, (i+n-1)^+\} \]
    \[ V\setminus A^-(i) = \{(i+1)^+, (i+2)^-, \dots, (i+n-1)^-\} \]
    for $1 \leq i \leq n$.

    We relabel them as $a_i=V\setminus A^+(i)$, $b_i=V\setminus A^-(i)$ for $1\leq i\leq n$. Since $A^+(i)$ has three independent subsets of cardinality $n-1$, which are exactly $a_{i-1},a_{i+1},b_i$. If we equate $a_i$ and $i^+$ in the circular ladder graph $CL_n$, and similarly equate $b_i$ and $i^-$, we see that $H_{n-1}\cong CL_n$ as desired.

    Therefore, it is sufficient to investigate the neighborhood complex of $CL_n$. Since the graph is $3$-regular, $\mathcal{N}(CL_n)$ is $2$-dimensional and has $2n$ facets: $F_i=\set{i^+,(i+1)^-,(i+2)^+}$ and $F_{n+i}=\set{i^-,(i+1)^+,(i+2)^-}$ for $1\le i\le n$. Notice that for all $1\le j\le 2n$, $|F_j\cap F_{j+1}|=2$, here we identify $F_{2n+1}=F_1$, see Figure \ref{fig:skeleton}. Moreover,$\set{i^+,(i+2)^+}$ and $\set{i^-,(i+2)^-}$ are free faces of certain facets. Using an element collapse argument, we get 
    \[
    \mathcal{N}(CL_n)\simeq \mathbb{S}^1.
    \]
\begin{figure}[ht]
\centering
\begin{tikzpicture}[
    vertex/.style={draw, circle, minimum size=10mm, fill=white, 
                   text height=1.5ex, text depth=.25ex},
    every edge/.style={draw, thick}
]

\foreach \i in {1,...,10} {
    \ifnum\i=1 \def\mylabel{$1^+$}\fi
    \ifnum\i=2 \def\mylabel{$2^-$}\fi
    \ifnum\i=3 \def\mylabel{$3^+$}\fi
    \ifnum\i=4 \def\mylabel{$4^-$}\fi
    \ifnum\i=5 \def\mylabel{$5^+$}\fi
    \ifnum\i=6 \def\mylabel{$1^-$}\fi
    \ifnum\i=7 \def\mylabel{$2^+$}\fi
    \ifnum\i=8 \def\mylabel{$3^-$}\fi
    \ifnum\i=9 \def\mylabel{$4^+$}\fi
    \ifnum\i=10 \def\mylabel{$5^-$}\fi
    
    \node[vertex] (v\i) at ({\i*36}:3cm) {\mylabel};
}

\foreach \i [evaluate={
    \next = int(mod(\i,10)+1);
    \nextnext = int(mod(\i+1,10)+1);
    \prevprev = int(mod(\i-2+10,10)+1);
}] in {1,...,10} {
    \draw (v\i) -- (v\next);
    \draw (v\i) -- (v\nextnext);
    \draw (v\i) -- (v\prevprev);
}

\end{tikzpicture}
\caption{$1$-skeleton of $\mathcal{N}(CL_5)$}\label{fig:skeleton}
\end{figure}
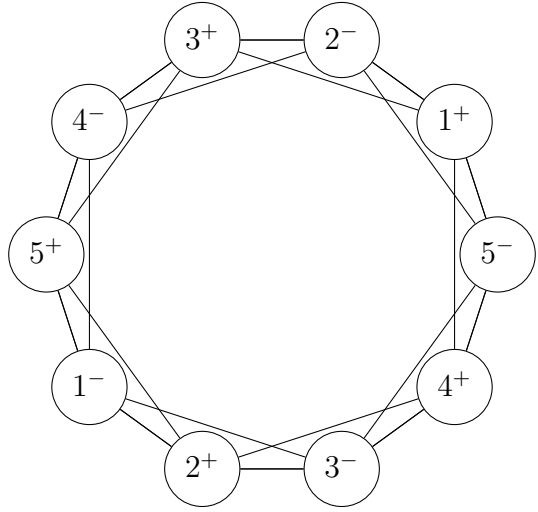
For $n$ to be even, $H_{n-1}$ has vertices $\binom{A}{n-1}$ and $\binom{B}{n-1}$, where $A$, $B$ has been defined above. Hence $\mathcal{N}(H_{n-1})$ has two disjoint simplices as facets:
    \begin{enumerate}
        \item $\{1^+, 2^-, \dots, n^-\}$,
        \item $\{1^-, 2^+, \dots, n^+\}$.
    \end{enumerate}
    which deduces that $\mathcal{N}(H_{n-1})\simeq \mathbb{S}^0$.
\end{proof}
The situation for squared cycle graphs is quite a bit more complicated. The \textit{squared cycle graph} $W_n$ is the graph with vertex set $[n]$, and edge set $\{\{i, i+1 \pmod n\}, \{i, i+2 \pmod n \}\}, i = 1, 2, \cdots, n$. We only check the case when $n=3k+1$ as below.
\begin{theorem}\cite{SSYZZ1}
    For $k\geq 3$, we have:
    \begin{equation}
        \Delta_k^t(W_{3k+1}) \simeq \mathbb{S}^3
    \end{equation}
\end{theorem}
\begin{theorem}\label{neighbor-Wn}
    Let $W_{3k+1}$ be the squared cycle graph of $3k+1$ vertices for $k\ge 3$. Let $H_k$ be an induced $k$-independent graph; then it is $k+2$-regular. Moreover, $\mathcal{N}(H_k)$ is $k+1$ dimensional and  homotopy equivalent to $\mathbb{S}^1$.
\end{theorem}
\begin{proof}

$H_k$ has $3k+1$ vertices, and we order them lexicographically. That is,
\begin{gather*}
v_1 = \{1, 4, \cdots, 3k-5, 3k-2\}, \quad
v_2 = \{1, 4, \cdots, 3k-5, 3k-1\}, \\
v_3 = \{1, 4, \cdots, 3k-4, 3k-1\}, \quad \cdots, \\
v_k = \{1, 5, \cdots, 3k-4, 3k-1\}, \quad
v_{k+1} = \{2, 5, \cdots, 3k-4, 3k-1\}, \\
v_{k+2} = \{2, 5, \cdots, 3k-4, 3k\}, \quad \cdots, \\
v_{2k} = \{2, 6, 9, \cdots, 3k-3, 3k\}, \quad
v_{2k+1} = \{3, 6, 9, \cdots, 3k-3, 3k\}, \\
v_{2k+2} = \{3, 6, 9, \cdots, 3k-3, 3k+1\}, \quad \cdots, \\
v_{3k} = \{3, 7, 10, \cdots, 3k-2, 3k+1\}, \quad
v_{3k+1} = \{4, 7, 10, \cdots, 3k-2, 3k+1\}.
\end{gather*}
After this ordering, $\mathcal{N}(H_k)$ has $3k+1$ facets characterized by each vertex, denoted by 
$F_i = \{i+k, i+k+1, \dots, i+2k+1\} \pmod{3k+1}, 1\le i\le 3k+1$. 

    Notice that for any $1 \leq i \leq 3k+1$, $F_i$ has a free face $\{i+k, i+2k+1\}$. We can remove these free faces without changing the homotopy type of $\mathcal{N}(H_k)$ by using an element collapse argument. After deleting these free faces, we get a simplicial complex with $3k+1$ facets: $\{1, 2, \cdots, k+1\}, \{2, \cdots, k+2\}, \cdots, \{3k+1, 1, 2, \cdots, k\}$. 
    By induction, $\mathcal{N}(H_k)$ is homotopy equivalent to a simplicial complex with $3k+1$ facets, $\{1, 2\}, \{2, 3\}, \cdots, \{3k+1, 1\}$, which is homotopy equivalent to $\mathbb{S}^1$. Therefore, we have $\mathcal{N}(H_k)\simeq \mathbb{S}^1$.
 
\end{proof}
Based on the above discussion, one might guess that the dimension of "the homotopy type", namely the spheres, of $\Delta_k^t(G)$ for a graph $G$ would be higher than that of $\mathcal{N}(H_k)$ induced by $G$. However, this is not the fact, due to the following counterexample.
\begin{ex}\label{counterexample}
    Consider the star graph $S_n$ (a central vertex $c$ connected to $n$ leaf vertices denoted by $[n]=\set{1,2,\dots,n}$). For $2\le k\le n$, the total cut complex $\Delta_k^t(S_n)$ is contractible, see [\cite{Bayer_2024_02}, Corollary 4.14].

    Now we investigate the induced $k$-independent graph. For $2\le k\le m$, $H_k(S_n)$ has vertices $\binom{[n]}{k}$; two of them are joined by an edge iff they are disjoint. In other words, $H_k$ is exactly the Kneser graph $KG(n,k)$. In \cite{Nilakantan}, the authors proved the homotopy type of $\mathcal{N}(H_2)$ is a wedge of $n^2-3n+1$ spheres of dimension $n-4$.
\end{ex}
In the following proposition, we prove a somewhat general result, which shows there is an open cover of $\mathcal{N}(H_k)$, with nerve complex $\Delta_k^t(G)$.
\begin{prop}\label{nerve-complex-general}
    Given a graph $G$, suppose that the independent number of $G$ is $\alpha(G)$. For $k\leq \alpha(G)$, let $H_k$ be induced $k$-independent graph. Then \(\Delta_k^t(G)\) is a nerve complex of an open cover (may NOT good) of $\mathcal{N}(H_k)$. 
\end{prop}
\begin{proof}
Suppose $G=(V(G),E(G))$, and $V(G)=[n]$. We may write $\mathcal{N}(H_k)$ as 
\[
\{F\subseteq \Delta_v:v \text{ is vertex of } \mathcal{N}(H_k)\}
\]
where $\Delta_v=\{w  \text{ is vertex of } \mathcal{N}(H_k) :v\cap w=\varnothing\}$.

For fixed $k\geq 1$ and any $1\leq i\leq n$, define
\[
A_i=\{ F\subseteq \Delta_v:v \text{ is vertex of } \mathcal{N}(H_k) \text{ such that } i\notin v\}.
\]

Then we see that 
\[
\mathcal{N}(H_k)=\bigcup_{i=1}^n A_i
\]

We show that $\Delta_k^t(G)$ is a nerve complex of $\mathcal{A}=\{A_i\}_{1\leq i\leq n}$. For an index set $I=\{i_1,i_2,\dots,i_m\}$, $\bigcap_{i_j\in I} A_{i_j}$ is nonempty if and only if there exists a vertex $v$ of $\mathcal{N}(H_k)$ such that $i_1,i_2,\dots,i_m\notin v$. Similarly to the previous proof, we obtain that it is equivalent to say $[n]\setminus I$ contains an independent set of size $k$. Notice the definition of the total cut complex, we conclude that $I$ is a face of $\Delta_k^t(G)$. Hence the nerve complex of $\mathcal{A}$ is exactly $\Delta_k^t(G)=\mathfrak{N}(\mathcal{A})$.   

\end{proof}
A key step in this proof is to construct the open cover $\{A_i\}$. Thus, in order to utilize the Nerve Lemma \ref{Nerve Lemma} and \ref{strong nerve lemma}, we need to prove that the non-empty intersection of these open sets satisfies some desirable properties. However, under normal circumstances, this matter is quite difficult, or the conclusion we want is simply not valid. We look forward to better methods that can more comprehensively reveal the relationship between these two complexes, thereby helping to solve some classic problems.
\section{Acknowledgments}

The authors thank the organizers of the 2025 PKU Algebraic Combinatorics Experience, where this work originated. The authors also thank Lei Xue, our mentor, who provided us with a lot of help.
We are also grateful to the anonymous referees for their careful reading of the paper.

\newpage

\renewcommand*{\bibfont}{\footnotesize}
    \bibliography{reftoneighbor}
    \bibliographystyle{apalike}

\end{document}